\documentclass[11pt]{amsart}
\usepackage{graphicx}
\usepackage{amssymb}
\newtheorem{tw}{Theorem}[section]

\theoremstyle{remark}

\newtheorem{ex}[tw]{Example}

\theoremstyle{definition}

\newcommand{\Cal}[1]{\mathcal{#1}}
\newcommand{\bez}{\setminus}

\newcommand{\sig}{\sigma}

\newcommand{\fal}[1]{\widetilde{#1}}

\newcommand{\gen}[1]{\langle #1 \rangle}
\newcommand{\map}[3]{#1\colon #2\to #3}
\newcommand{\field}[1]{\mathbb{#1}}
\newcommand{\zz}{\field{Z}}

\newcommand{\rr}{\field{R}}

\begin{document}

\numberwithin{equation}{section}
\title[Subgroups of the Torelli group generated by two symmetric\ldots]
{Subgroups of the Torelli group generated by two symmetric bounding pair maps}

\author{Micha\l\ Stukow}

\thanks{Supported by grant 2015/17/B/ST1/03235 of National Science Centre, Poland.}
\address[]{
Institute of Mathematics, University of Gda\'nsk, Wita Stwosza 57, 80-952 Gda\'nsk, Poland }

\email{trojkat@mat.ug.edu.pl}




\begin{abstract}
Let $\{a,b\}$ and $\{c,d\}$ be two pairs of bounding simple closed curves on an oriented surface which intersect nontrivialy. We prove that if these pairs are invariant under the action of an orientation reversing involution, then the corresponding bounding pair maps generate a free group. This supports the conjecture stated by C.~Leininger and D.~Margalit that any pair of elements of the Torelli group either commute or generate a free group.
\end{abstract}

\maketitle%
 \section{Introduction}%
Let ${\Cal{M}}(S_g)$ be the \emph{mapping class group} of a closed oriented surface $S_g$ of genus $g$, that is the group of isotopy classes of orientation preserving diffeomorphisms of $S_g$. 
If we include isotopy classes of orientation reversing maps, we obtain the so--called \emph{extended mapping class group} ${\Cal{M}}^{\pm}(S_g)$. 
The mapping class group acts on $\textrm{H}_1(S_g;\zz)$ and the kernel of this action ${\Cal{I}}(S_g)$ is called the \emph{Torelli subgroup} of ${\Cal{M}}(S_g)$. The Torelli subgroup plays an important role in a study of mapping class group -- see for example Chapter~8 of \cite{MargaliFarb} and references there. 

The Torelli group 
is generated by two types of elements: by \emph{Dehn twists} $t_a$ about separating simple closed curves, and by \emph{bounding pair maps}, which are defined as products of Dehn twists $t_at_b^{-1}$ about disjoint and homologous simple closed curves \cite{Powell}. This generating set can be further simplified,  for example it is known \cite{JohnAnn} that if $g\geq 3$, then  ${\Cal{I}}(S_g)$ is generated by a finite number of bounding pair maps. 

As for relations, much less is known. For example, the only relations we know between two elements of ${\Cal{I}}(S_g)$ are obvious commutativity relations. Hence C.~Leininger and D.~Margalit conjectured \cite{MargLein} that any pair of elements of the Torelli group either commute or generate a free group. In fact, they showed in \cite{MargLein} that the pure braid group has such property. It is known \cite{Ishida} that the Leininger--Margalit conjecture is true for Dehn twists about separating simple closed curves, so the next natural question is if the conjecture is true for bounding pair maps. The goal of this paper is to show that there is a large class of bounding pair maps which support the Leininger--Margalit conjecture -- see Theorem \ref{Main:thm}.
\section{Symmetric bounding pairs}
Following \cite{RolPar}, we say that simple closed curve $a$ in $S_g$ is \emph{essential} if it does not bound a disk. We say that $M$ is an \emph{essential subsurface} of $S_g$ if $M$ is a closed connected subsurface of $S_g$
and each component of $\partial M$ is essential in $S_g$.
We say that $\{a,b\}$ is a \emph{bounding pair} in $S_g$, 
if $a$ and $b$ are two disjoint, homologous, and nonisotopic simple closed curves in $S_g$. 
We say that $\{a,b\}$ is a bounding pair in an essential subsurface $M\subset S_g$, if $\{a,b\}$ is a bounding pair in $S_g$ and $a,b$ are contained in the interior of $M$. Observe that $S_g$ admits bounding pairs only if $g>2$. 

Let 
$\map{\sig}{M}{M}$ be an orientation reversing involution of an essential subsurface $M\subset S_g$, and assume that $\{a,b\}$ is bounding pair in $M$.
If $\sig(a)=b$, then we say that $(\sig,M)$ is a \emph{symmetry} of the bounding pair $\{a,b\}$. We say that two bounding pairs $\{a,b\}$ and $\{c,d\}$ are \emph{symmetric}, if they have a common symmetry $(\sig,M)$.
\begin{ex}\label{Ex:1}
If $\{a,b\}$ is any bounding pair in $S_g$, then $a$ and $b$ bound an essential subsurface $M_{a,b}$. Moreover, we can slightly enlarge $M_{a,b}$ and assume that $a,b$ are contained in the interior of $M_{a,b}$. Hence, if $\map{\sig}{M_{a,b}}{M_{a,b}}$ is any orientation reversing involution which interchange $a$ and $b$, then $(\sig,M_{a,b})$ is a symmetry of $\{a,b\}$. 

Moreover, by the classical results of Harnack \cite{Harnack} and Weichold \cite{Weichold}, we know the full topological classification of such orientation reversing involutions (symmetries) of $M_{a,b}$. If $M_{a,b}$ has genus $h$, then there are $\left\lceil\frac{h+1}{2}\right\rceil$ symmetries for which the complement $M_{a,b}\bez \rm{Fix}(\sig)$ of the set $\rm{Fix}(\sig)$ of fixed points of $\sig$ is disconnected (so--called \emph{separating} symmetries), and $h+1$ symmetries for which $M_{a,b}\bez \rm{Fix}(\sig)$ is connected (so--called \emph{nonseparating} symmetries).
In order to explain their geometric interpretation, assume that $M_{a,b}$ is embedded in $\rr^3$ in a symmetric manner as shown in Figure~\ref{r01}. 
\begin{figure}[h]
\begin{center}
\includegraphics[width=0.9\textwidth]{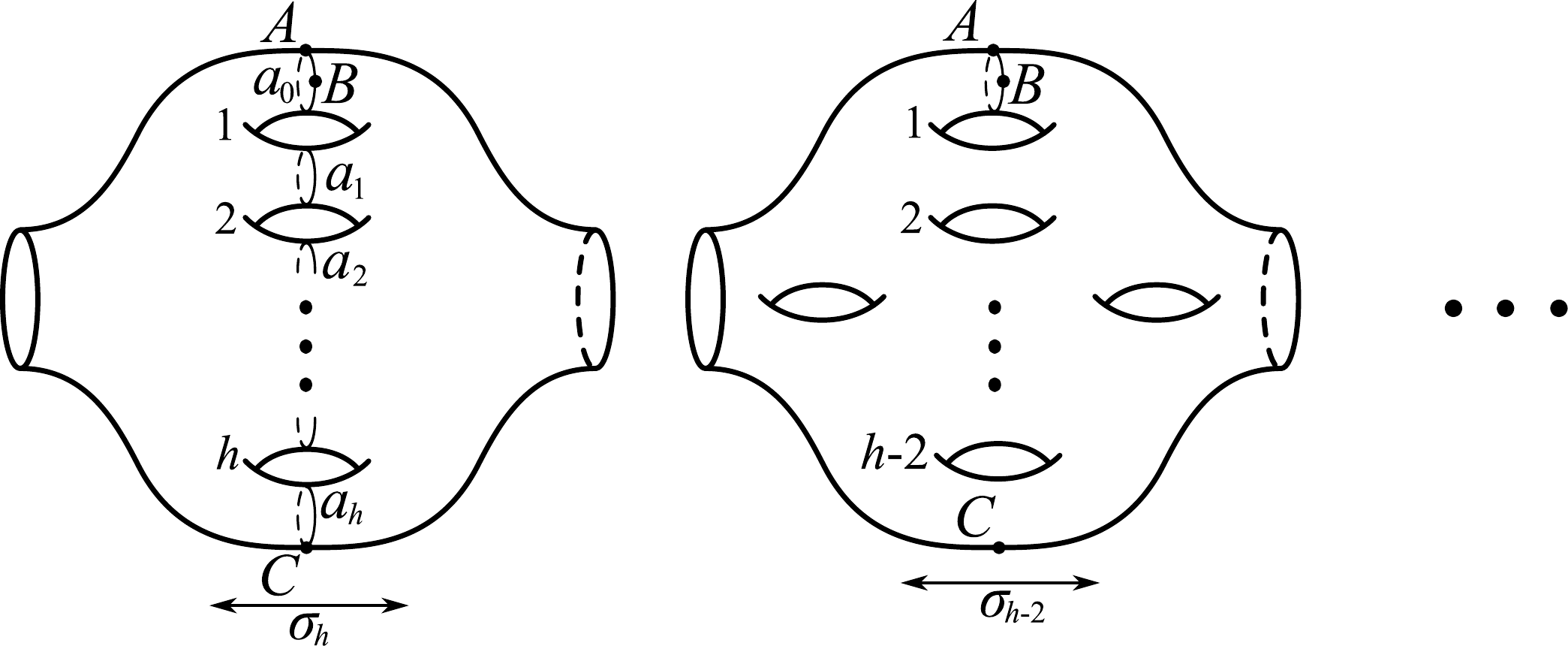}
\caption{Symmetries $\sig_h,\sig_{h-2},\ldots$ of a subsurface $M_{a,b}$.}\label{r01} %
\end{center}
\end{figure}
In this figure, separating symmetries $\sig_{h},\sig_{h-2},\ldots$ are symmetries across the vertical plane containing the points $\{A,B,C\}$. If $a_0,a_1,\ldots,a_h$ are simple closed curves as in Figure \ref{r01}, then the nonseparating symmetries can be defined as $\tau_i=\sig_ht_{a_h}t_{a_{h-1}}\cdots t_{a_{i}}$ for $i=0,1,\ldots,h$. It is straightforward to check that $\rm{Fix}(\tau_i)$ consists of $i$ circles: $\{a_0,a_1,\ldots,a_{i-1}\}$. Observe also that the orbit space $M_{a,b}/\gen{\sig}$ is orientable if and only if $\sig$ is separating.
\end{ex}
If $\{a,b\}$ and $\{c,d\}$ are symmetric bounding pairs in $S_g$, then we define their \emph{intersection number} as
\[I(\{a,b\},\{c,d\})=I(a,c)+I(b,c),\]
where on the right hand side we have the usual geometric intersection number of simple closed curves. If $I(\{a,b\},\{c,d\})=0$ then
the corresponding boundary pair maps $t_{a}t_b^{-1}$ and $t_ct_{d}^{-1}$ commute, hence we are interested in the case $I(\{a,b\},\{c,d\})>0$. 
Observe that since $a$ and $b$ bound a subsurface of $S_g$, $I(\{a,b\},\{c,d\})$ is always even.
\begin{ex}
 Using the analysis made in Example \ref{Ex:1} it is straightforward to construct interesting examples of symmetric bounding pairs $\{a,b\}$ and $\{c,d\}$. For example Figure \ref{r02} shows how to construct a symmetric pair of bounding pairs $\{a,b\}$ and $\{c,d\}$ such that $I(\{a,b\},\{c,d\})=4$. 
 \begin{figure}[h]
\begin{center}
\includegraphics[width=0.8\textwidth]{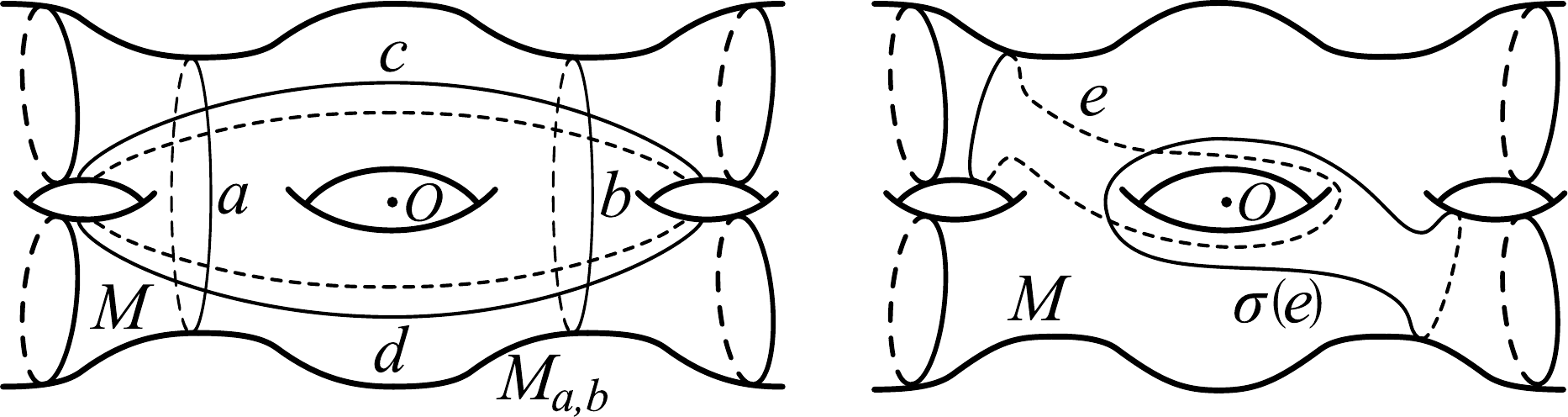}
\caption{Bounding pairs on subsurface $M$.}\label{r02} %
\end{center}
\end{figure}
 The symmetry we use in this figure is a point reflection $\map{\sig}{M}{M}$ through the symmetry centre $O$ of $M_{a,b}$
 
 Moreover, if $e$ is a simple closed curve indicated in this figure and $n>0$, then, by the well know formula \cite{Ishida} for the geometric intersection number, we have
 \[I(t_e^n(c),a)\geq nI(e,c)I(e,a)-I(c,a)=4n-2.\]
 Hence we can construct symmetric bounding pairs $\{a,b\},\{t_e^{n}(c),t_{\sig(e)}^{-n}(d)\}$ with arbitrary large intersection number $I(\{a,b\},\{t_e^{n}(c),t_{\sig(e)}^{-n}(d)\})$. 
 
 Observe also that our definition of a bounding pair $\{a,b\}$ implies that both $M_{a,b}$ and $S_g\bez M_{a,b}$ has genus at least 1, hence the above construction can be applied to any arbitrary chosen bounding pair $\{a,b\}$. 
\end{ex}
Now we are ready to state the main theorem of the paper.
\begin{tw}\label{Main:thm}
 Let $\{a,b\}$ and $\{c,d\}$ be symmetric bounding pairs in $S_g$, and assume that $I(\{a,b\},\{c,d\})>0$. Then the corresponding bounding pairs maps $t_at_b^{-1}$ and $t_ct_d^{-1}$ generate a free subgroup of ${\Cal{I}}(S_g)$.
\end{tw}
\begin{proof}
 Without loss of generality we can assume that 
 \[\begin{aligned}
 &|a\cap c|=I(a,c)=I(b,d)=|b\cap d|\\
 &|a\cap d|=I(a,d)=I(b,c)=|b\cap c|.
 \end{aligned}\]
 Since $I(\{a,b\},\{c,d\})$ is even, then 
 \[I(\{a,b\},\{c,d\})=I(a,c)+I(b,c)\geq 2.\]
 Let $\map{\pi_\sig}{M}{M/\gen{\sig}}$ be the canonical projection on the orbit space. 
 
 Since $d=\sig(c)$ and $c$ are disjoint, $c$ is disjoint from the set of fixed points of $\sig$, and $\sig$ is a bijection between the sets $c\cap(a\cup b)$ and $d\cap (a\cup b)$. In particular
 \[|\pi_\sig(a)\cap \pi_\sig(c)|=I(a,c)+I(a,d)=I(a,c)+I(b,c)\geq 2.\]
 We now argue that $\pi_\sig(a)$ and $\pi_\sig(c)$ are in a minimal position with respect to the intersection number. Suppose to the contrary that there is a disk $\Delta$ bounded by two arcs $p,q$ of $\pi_\sig(a)$ and $\pi_\sig(c)$ respectively. By taking the inner most such disk we can assume that there are no intersection points of $\pi_\sig(a)$ and $\pi_\sig(c)$ between the endpoints $P,Q$ of $p$ and $q$. Let $\fal{\Delta}$ be a lift of $\Delta$ with respect to $\pi_\sig$ such that the lift $\fal{p}$ of $p$ is an arc of $a$, and let $\fal{q}$ be the corresponding lift of $q$. There are 3 possible cases: either $\fal{q}$ connects two intersection points of $a\cap c$, or $\fal{q}$ connects two intersection points of $a\cap d$, or else $\fal{q}$ connects intersection points of $a\cap c$ and $b\cap c$. In the first two cases we get a contradiction with the assumption that $c$ and $d$ are in a minimal position with $a$ and $b$, and the third case (see Figure \ref{r03}) would imply that $p\cup q$ is an one-sided circle (since $\sig$ is orientation reversing). 
 \begin{figure}[h]
\begin{center}
\includegraphics[width=0.28\textwidth]{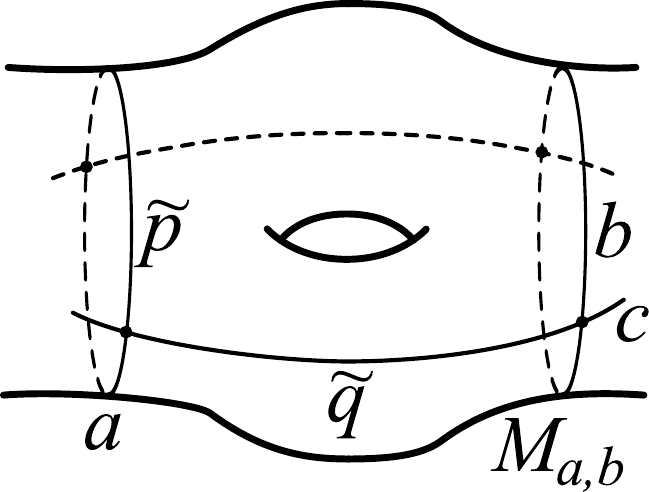}
\caption{One of the possible lifts of $p\cup q$.}\label{r03} %
\end{center}
\end{figure}
 Hence we proved that 
 \[I(\pi_\sig(a),\pi_\sig(c))=|\pi_\sig(a)\cap \pi_\sig(c)|\geq 2.\]
 By Theorems 1.2 of \cite{Ishida} and 13.2 of \cite{StukowTwoTwists}, 
 $t_{\pi_\sig(a)}$ and $t_{\pi_\sig(c)}$ 
 generate a free subgroup of ${\Cal{M}}(M/\gen{\sig})$ (note that $M/\gen{\sig}$ is nonorientable if $\sig$ is nonseparating). Moreover, for each $f\in {{\Cal{M}}(M/\gen{\sig})}$ we can choose an orientation preserving lift $s_{\sig}(f)=\fal{f}\in {{\Cal{M}}^{\pm }(M)}$ of $f$, which defines a section $\map{s_\sig}{{\Cal{M}}(M/\gen{\sig})}{{\Cal{M}}^{\pm }(M)}$ of the projection $\map{(\pi_\sig)_*}{{\Cal{M}}^{\pm }(M)}{{\Cal{M}}(M/\gen{\sig})}$.
 Hence $t_at_b^{-1}=s_{\sig}(t_{\pi_\sig(a)})$ and $t_ct_d^{-1}=s_{\sig}(t_{\pi_\sig(c)})$ generate a free subgroup of ${\Cal{M}}(M)$. 
 
 It remains to show that these two elements generate a free subgroup of ${\Cal{M}}(S)$. By Theorem 4.1 of \cite{RolPar}, the kernel of the inclusion \[\map{i_*}{{\Cal{M}}(M)}{{\Cal{M}}(S)}\]
 is generated by elements which are central in ${\Cal{M}}(M)$. Hence if $k$ is an element of the kernel of 
 \[\map{i_*}{\gen{t_at_b^{-1},t_ct_d^{-1}}}{{\Cal{M}}(S)},\]
 then $k$ is central in the free group $\gen{t_at_b^{-1},t_ct_d^{-1}}$, hence $k=1$.
\end{proof}
\section*{Acknowledgements}
The author wishes to thank Dan Margalit, who first observed that there is a connection between our main result in \cite{StukowTwoTwists} and the Leininger--Margalit conjecture.




\begin{thebibliography}{1}

\bibitem{MargaliFarb}
{\sc B.~Farb and D.~Margalit}, {A Primer on Mapping Class Groups}, 
  Princeton Mathematical Series {\bf 49}, Princeton Univ. Press, 2011.

\bibitem{Harnack}
{\sc A.~Harnack}, {{\"U}ber die {V}ieltheiligkeit der ebenen algebraischen
  {K}urven}, Math. Ann. {\bf 10} (1876), 189--199.

\bibitem{Ishida}
{\sc A.~{I}shida}, {The structure of subgroups of mapping class groups
  generated by two {D}ehn twists}, Proc. Japan Acad. Ser. A Math. Sci. {\bf 72}
  (1996), 240--241.

\bibitem{JohnAnn}
{\sc D.~L. Johnson}, {The structure of the {T}orelli group {I}: A finite
  set of generators for {${\Cal{I}}$}}, Ann. of Math. {\bf 118} (1983),
  423--442.

\bibitem{MargLein}
{\sc C.~J. {L}eininger and D.~Margalit}, {Two--generator subgroups of the
  pure braid group}, Geom. Dedicata {\bf 147} (2010), 107--113.

\bibitem{RolPar}
{\sc L.~{P}aris and D.~{R}olfsen}, {Geometric subgroups of mapping class
  groups}, J. Reine Angew. Math. {\bf 521} (2000), 47--83.

\bibitem{Powell}
{\sc J.~Powell}, {Two theorems on the mapping class group of a surface},
  Proc. Amer. Math. Soc. {\bf 68} (1978), 347--350.

\bibitem{StukowTwoTwists}
{\sc M.~{S}tukow}, {Subgroups generated by two {D}ehn twists on
  nonorientable surface}.
\newblock arXiv:1310.3033v3 [math.GT], 2013.

\bibitem{Weichold}
{\sc G.~Weichold}, {{\"U}ber symmetrische {R}iemannsche {F}l{\"a}chen und
  die {P}eriodizit{\"a}tsmodulen der zugerh{\"o}rigen {A}belschen
  {N}ormalintegrale erstes {G}attung}, Zeitschrift f{\"u}r Math. und Phys. {\bf 28}
  (1883), 321--351.

\end{thebibliography}
%
%

\end{document}